\documentclass{amsart}
\usepackage{amsmath,amsthm,amssymb,IMjournal}
\usepackage{pgf, tikz}
\usepackage{pstricks-add}

\begin{document}
\newtheorem{The}{Theorem}[section]
\newtheorem{prop}[The]{Proposition}
\newtheorem{lem}[The]{Lemma}
\newtheorem{cor}[The]{Corollary}

\newtheorem{rem}[The]{Remark}
\newtheorem{example}[The]{Example}

\numberwithin{equation}{section}

\title{A note on rational surgeries on a Hopf link}

\author{\|Velibor |Bojkovi\' c|, Caen,
        \|Jovana |Nikoli\' c|, Beograd, 
        \|Mladen |Zeki\' c|, Beograd}

\abstract 
It is clear that every rational surgery on a Hopf link in $3$-sphere is a lens space surgery. In this note we give an explicit computation which lens space is a resulting manifold. The main tool we use is the calculus of continued fractions. As a corollary, we recover the (well known) result on the criterion for when rational surgery on a Hopf link gives the $3$-sphere.
\endabstract



\thanks
   The second and third authors were supported by the Science Fund of the Republic of Serbia, GRANT No 7749891, Graphical Languages - GWORDS.
\endthanks

\section{Introduction}

It is well known that every closed, orientable 3-manifold can be obtained by Dehn surgery on a framed link in $\mathbb{S}^3$. In \cite{Kir}, Kirby gave an answer when two framed links in $\mathbb{S}^3$ produce homeomorphic 3-manifolds, using what is nowdays called Kirby integral calculus (which pertains to surgery with integral coefficients). However, preceding these results, Rolfsen in \cite{R76} provided more general framework for surgery on links with rational coefficients (to which refer in this paper as  rational or Rolfsen calculus) and in \cite{Rolf} generalized the main result of \cite{Kir}. The advantage of using surgery with rational coefficients is that we can simplify surgery presentation of a 3-manifold. For example, the homology sphere resulting from surgery on a trefoil with coefficient $\frac{1}{n}$, where $n$ is an integer, can also be constructed using an integral surgery but one has to work with more complicated links (see Example in~\cite{Rolf} and Figure \ref{Figure:ComplicatedLinks}). For more details about surgery of 3-manifolds and Kirby calculus see \cite[Chapter VI]{PS96}, and for presentation of the rational surgery we refer to \cite[Section 9.H]{R76} and \cite{Rolf}.

Let us recall that a surgery along a link is called \textit{a lens space surgery} if the result is a lens space (we will consider manifolds $\mathbb{S}^3$ and $\mathbb{S}^2\times \mathbb{S}^1$ as (trivial) special cases of lens spaces and denote them by $L(1,0)$ and $L(0,1)$, respectively). In this note a Hopf link denotes two-component link constisting of two unknots linked together once. It is denoted as $2^2_1$ link in \cite[Appendix C, Table of knots and links]{R76} and sometimes also refered to as $L2a1$ link.

It is known that a surgery on a Hopf link with framing $m\in\mathbb{Z}$ and $\frac{p}{q}\in\mathbb{Q}$ (see Figure \ref{Figure:HopfLink}) is a lens space surgery and the 3-manifold we obtain is $L(a,b)$ where $\frac{a}{b}=m-\frac{q}{p}$ (see Proposition 17.3. in~\cite{PS96} for more general statement). Note that this transformation of a $2$-component surgery to a $1$-component surgery (with coefficients as above) is well known under the name \textit{slam-dunk move} (see \cite[The Slam-Dunk Theorem]{CG88} and \cite[Section 5.3]{GS99}). 
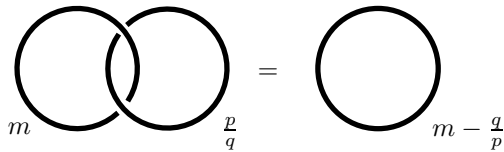
\begin{figure}[!ht]
\centering 
\psset{xunit=0.4cm,yunit=0.4cm,algebraic=true,dimen=middle,dotstyle=o,dotsize=3pt 0,linewidth=0.8pt,arrowsize=3pt 2,arrowinset=0.25}
\begin{pspicture*}(-3.49,-2)(15.5,3.4)
\parametricplot[linewidth=2pt]{-3.758972595266308}{2.321589552819717}{1*1.98*cos(t)+0*1.98*sin(t)+4|0*1.98*cos(t)+1*1.98*sin(t)+1}
\parametricplot[linewidth=2pt]{-0.5735991128384956}{5.457219563177892}{1*2*cos(t)+0*2*sin(t)+1|0*2*cos(t)+1*2*sin(t)+1}
\pscircle[linewidth=2pt](11,1){0.8}
\rput[tl](-1.3,-0.8){$m$}
\rput[tl](7,1){$=$}
\rput[tl](5.8,-0.6){$\frac{p}{q}$}
\rput[tl](12.8,-0.6){$m-\frac{q}{p}$}
\end{pspicture*}
\caption{Hopf link with framing $m\in\mathbb{Z}$ and $\frac{p}{q}\in\mathbb{Q}$} 
\label{Figure:HopfLink}
\end{figure}

We generalize this result to the case when surgery coefficients on both components of a Hopf link are rational numbers (see Figure \ref{Figure:HopfLinkRational1}). To state our main result and for the purposes of this note, we will say that \textit{a continued fraction} is an expression of the form $[a_0;a_1,\dots,a_l]$, where $a_0,\dots, a_l\in \mathbb{Z}$. If we adjoin the symbol $\frac{1}{0}$ (which we will sometimes denote by $\infty$) to $\mathbb{Q}$ and define $\frac{1}{0}+r=\frac{1}{0}$ and $\frac{r}{\frac{1}{0}}=0$, $r\in \mathbb{Q}$), then any continued fraction as before represents an element $r$ (we write $r=[a_0;a_1,\dots,a_l]$), where $r\in \mathbb{Q}\cup\{\frac{1}{0}\}$ is defined as:
\begin{enumerate}
\item[(1)] If $l=0$, then $r=a_0$;
\item[(2)] If $l>0$, then $r=a_0+\frac{1}{[a_1;\dots,a_l]}$.
\end{enumerate}
Then, our main result is:

\begin{The}\label{Theorem:Basic} Rational surgery on a Hopf link with framing $\frac{p}{q}=[a_0;a_1,\hdots,a_n]$ and $\frac{r}{s}=[b_0;b_1,\hdots;b_m+1]$ is a lens space surgery and the resulting lens space is $L(a,b)$ where 
$$\frac{a}{b}=[-b_m;\hdots,-b_0, a_0,\hdots,a_n]^{(-1)^{m+1}}-1,$$
(if $\frac{a}{b}\in\mathbb{Q}$, we take $a$ and $b$  coprime, $b>0$).
\end{The}

\vspace{-0.3cm}
\begin{figure}[!ht]
\psset{xunit=0.4cm,yunit=0.4cm,algebraic=true,dimen=middle,dotstyle=o,dotsize=3pt 0,linewidth=0.8pt,arrowsize=3pt 2,arrowinset=0.25}
\begin{pspicture*}(-3.49,-1.5)(11,3.2)
\parametricplot[linewidth=2pt]{-3.758972595266308}{2.321589552819717}{1*1.98*cos(t)+0*1.98*sin(t)+4|0*1.98*cos(t)+1*1.98*sin(t)+1}
\parametricplot[linewidth=2pt]{-0.5735991128384956}{5.457219563177892}{1*2*cos(t)+0*2*sin(t)+1|0*2*cos(t)+1*2*sin(t)+1}
\rput[tl](-1.3,-0.2){$\frac{p}{q}$}
\rput[tl](5.8,-0.2){$\frac{r}{s}$}
\rput[tl](6.4,1.25){$=$}
\rput[tl](7.4,1.6){$L(a,b)$}
\end{pspicture*}
\caption{Hopf link with rational surgery on both unknots} 
\label{Figure:HopfLinkRational1}
\end{figure}
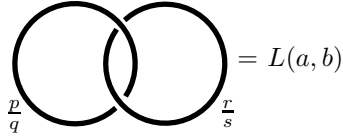

\begin{rem} {\rm The numbers $a$ and $b$ are not uniquely determined by $\frac{p}{q}$ and $\frac{r}{s}$. For example, if we take $\frac{p}{q}=2=[2]$ and $\frac{r}{s}=\frac{3}{2}=[1;2]=[2;-2]$, applying the theorem for the two representations of $\frac{r}{s}$, we obtain $\frac{a}{b}=\frac{-4}{1}$ and $\frac{a}{b}=\frac{4}{3}$, respectively.}
\end{rem}

Note that a rational surgery along a Hopf chain is not necessarily a lens space surgery. The example we mentioned in the beginning says that the homology sphere obtained via $\frac{1}{n}-$surgery along trefoil in $\mathbb{S}^3$ (which is not a lens space) is homeomorphic to the manifold obtained via surgery on a 3-component Hopf chain with coefficients $-3,-\frac{1}{2}$ and $\frac{1-6n}{n}$ (see Figure \ref{Figure:ComplicatedLinks} and Example in~\cite{Rolf}). 

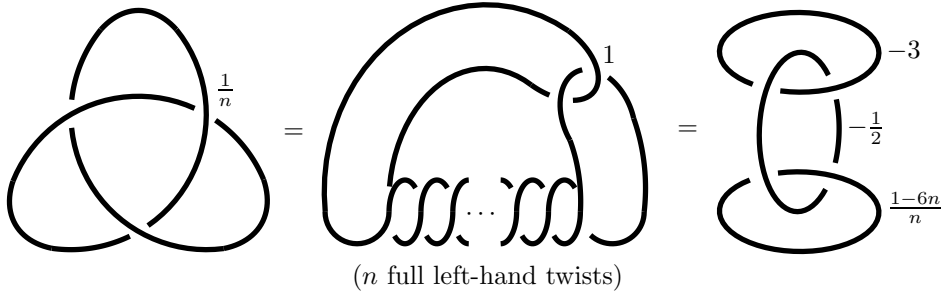
\begin{figure}[!ht]
\begin{center}
\begin{tikzpicture}[scale=0.6][line cap=round,line join=round,x=1.0cm,y=1.0cm]
\draw [shift={(0.9999999999999987,-0.999999999999996)},line width=2.0pt]  plot[domain=0.3939297389921699:0.9660899232838431,variable=\t]({1.0*2.999999999999997*cos(\t r)+-0.0*2.999999999999997*sin(\t r)},{0.0*2.999999999999997*cos(\t r)+1.0*2.999999999999997*sin(\t r)});
\draw [shift={(0.9999999999999987,-0.999999999999996)},line width=2.0pt]  plot[domain=1.144596485081734:2.760182933829636,variable=\t]({1.0*2.999999999999997*cos(\t r)+-0.0*2.999999999999997*sin(\t r)},{0.0*2.999999999999997*cos(\t r)+1.0*2.999999999999997*sin(\t r)});
\draw [shift={(-0.5000000000000022,1.598076211353317)},line width=2.0pt]  plot[domain=-0.9646773216065849:0.678318669034815,variable=\t]({1.0*3.0000000000000027*cos(\t r)+-0.0*3.0000000000000027*sin(\t r)},{0.0*3.0000000000000027*cos(\t r)+1.0*3.0000000000000027*sin(\t r)});
\draw [shift={(2.5000000000000013,1.598076211353317)},line width=2.0pt]  plot[domain=2.463273984554978:3.0332029277416814,variable=\t]({1.0*3.000000000000001*cos(\t r)+-0.0*3.000000000000001*sin(\t r)},{0.0*3.000000000000001*cos(\t r)+1.0*3.000000000000001*sin(\t r)});
\draw [shift={(2.5000000000000013,1.598076211353317)},line width=2.0pt]  plot[domain=3.2415543095310193:4.8420580169908165,variable=\t]({1.0*3.000000000000001*cos(\t r)+-0.0*3.000000000000001*sin(\t r)},{0.0*3.000000000000001*cos(\t r)+1.0*3.000000000000001*sin(\t r)});
\draw [shift={(-0.5000000000000022,1.598076211353317)},line width=2.0pt]  plot[domain=4.570199924546549:5.162555393670436,variable=\t]({1.0*3.0000000000000027*cos(\t r)+-0.0*3.0000000000000027*sin(\t r)},{0.0*3.0000000000000027*cos(\t r)+1.0*3.0000000000000027*sin(\t r)});
\draw [shift={(1.0000222232937448,1.3021201651520649)},line width=2.0pt]  plot[domain=2.563370686816426:3.719796285787093,variable=\t]({1.5591629334087065E-5*2.6012710786393898*cos(\t r)+0.9999999998784506*1.5294482208584166*sin(\t r)},{-0.9999999998784506*2.6012710786393898*cos(\t r)+1.5591629334087065E-5*1.5294482208584166*sin(\t r)});
\draw [shift={(0.30931227850565135,0.3337979345566995)},line width=2.0pt]  plot[domain=2.524624488412149:3.7580739377417633,variable=\t]({0.8658323388461026*2.3557691775981904*cos(\t r)+-0.5003342492857029*1.4857211778984538*sin(\t r)},{0.5003342492857029*2.3557691775981904*cos(\t r)+0.8658323388461026*1.4857211778984538*sin(\t r)});
\draw [shift={(1.6663976313762991,0.34698361637494374)},line width=2.0pt]  plot[domain=2.515726139911581:3.7671484886127042,variable=\t]({-0.8661474875708983*2.3684362351760284*cos(\t r)+-0.49978848503604034*1.5064796143369639*sin(\t r)},{0.49978848503604034*2.3684362351760284*cos(\t r)+-0.8661474875708983*1.5064796143369639*sin(\t r)});
\draw (2.5,2.7) node[anchor=north west] {$\frac{1}{n}$};
\draw (4,1.5) node[anchor=north west] {$=$};
\draw (2,-2.1) node[anchor=north west] {$\;$};
\end{tikzpicture} 
\begin{tikzpicture}[scale=0.212][line cap=round,line join=round,x=1.0cm,y=1.0cm]
\draw [shift={(9.0,1.0)},line width=2.0pt]  plot[domain=-0.7227005154289552:1.3495276921745412E-4,variable=\t]({6.747638431685932E-5*2.0000000136591907*cos(\t r)+-0.9999999977234688*0.9999999982926034*sin(\t r)},{0.9999999977234688*2.0000000136591907*cos(\t r)+6.747638431685932E-5*0.9999999982926034*sin(\t r)});
\draw [shift={(9.0,1.0)},line width=2.0pt]  plot[domain=3.1417276063590105:4.71242271857824,variable=\t]({6.747638431685932E-5*2.0000000136591907*cos(\t r)+-0.9999999977234688*0.9999999982926034*sin(\t r)},{0.9999999977234688*2.0000000136591907*cos(\t r)+6.747638431685932E-5*0.9999999982926034*sin(\t r)});
\draw [shift={(11.0,1.0000000000000002)},line width=2.0pt]  plot[domain=-0.7227005154289126:1.5708300649851559,variable=\t]({6.747638434322898E-5*2.0000000136591654*cos(\t r)+-0.9999999977234688*0.9999999982925908*sin(\t r)},{0.9999999977234688*2.0000000136591654*cos(\t r)+6.747638434322898E-5*0.9999999982925908*sin(\t r)});
\draw [shift={(11.0,1.0000000000000002)},line width=2.0pt]  plot[domain=2.4133772617930163:4.712422718574949,variable=\t]({6.747638434322898E-5*2.0000000136591654*cos(\t r)+-0.9999999977234688*0.9999999982925908*sin(\t r)},{0.9999999977234688*2.0000000136591654*cos(\t r)+6.747638434322898E-5*0.9999999982925908*sin(\t r)});
\draw [shift={(13.0,0.9999999999999999)},line width=2.0pt]  plot[domain=2.4117708026239186:4.712422718568368,variable=\t]({6.747638216353608E-5*2.000000013659158*cos(\t r)+-0.999999997723469*0.9999999982925872*sin(\t r)},{0.999999997723469*2.000000013659158*cos(\t r)+6.747638216353608E-5*0.9999999982925872*sin(\t r)});
\draw [shift={(13.0,0.9999999999999999)},line width=2.0pt]  plot[domain=-0.7227005154307786:1.5708300649785745,variable=\t]({6.747638216353608E-5*2.000000013659158*cos(\t r)+-0.999999997723469*0.9999999982925872*sin(\t r)},{0.999999997723469*2.000000013659158*cos(\t r)+6.747638216353608E-5*0.9999999982925872*sin(\t r)});
\draw [shift={(3.0000000000000004,0.9999999999999996)},line width=2.0pt]  plot[domain=-0.7227005154289454:1.570830064995028,variable=\t]({6.747638432584898E-5*2.00000001365919*cos(\t r)+-0.9999999977234688*0.999999998292603*sin(\t r)},{0.9999999977234688*2.00000001365919*cos(\t r)+6.747638432584898E-5*0.999999998292603*sin(\t r)});
\draw [shift={(3.0000000000000004,0.9999999999999996)},line width=2.0pt]  plot[domain=2.42064245570911:4.71242271857824,variable=\t]({6.747638432584898E-5*2.00000001365919*cos(\t r)+-0.9999999977234688*0.999999998292603*sin(\t r)},{0.9999999977234688*2.00000001365919*cos(\t r)+6.747638432584898E-5*0.999999998292603*sin(\t r)});
\draw [shift={(5.000000000000001,1.0000000000000002)},line width=2.0pt]  plot[domain=-0.7227005154289472:1.5708300649884466,variable=\t]({6.747638432644829E-5*2.0000000136591924*cos(\t r)+-0.9999999977234688*0.9999999982926043*sin(\t r)},{0.9999999977234688*2.0000000136591924*cos(\t r)+6.747638432644829E-5*0.9999999982926043*sin(\t r)});
\draw [shift={(5.000000000000001,1.0000000000000002)},line width=2.0pt]  plot[domain=2.4190228062475083:4.71242271857824,variable=\t]({6.747638432644829E-5*2.0000000136591924*cos(\t r)+-0.9999999977234688*0.9999999982926043*sin(\t r)},{0.9999999977234688*2.0000000136591924*cos(\t r)+6.747638432644829E-5*0.9999999982926043*sin(\t r)});
\draw [shift={(7.0,1.0000000000000002)},line width=2.0pt]  plot[domain=1.3495276933059384E-4:1.5708300649851559,variable=\t]({6.747638432405105E-5*2.0000000136591836*cos(\t r)+-0.9999999977234688*0.9999999982925998*sin(\t r)},{0.9999999977234688*2.0000000136591836*cos(\t r)+6.747638432405105E-5*0.9999999982925998*sin(\t r)});
\draw [shift={(7.0,1.0000000000000002)},line width=2.0pt]  plot[domain=2.4174061234267983:3.141727606359046,variable=\t]({6.747638432405105E-5*2.0000000136591836*cos(\t r)+-0.9999999977234688*0.9999999982925998*sin(\t r)},{0.9999999977234688*2.0000000136591836*cos(\t r)+6.747638432405105E-5*0.9999999982925998*sin(\t r)});
\draw [shift={(-0.006082265080781609,1.0060822650807817)},line width=2.0pt]  plot[domain=3.144643053367686:6.280153404375947,variable=\t]({1.0*2.0060914855061203*cos(\t r)+-0.0*2.0060914855061203*sin(\t r)},{0.0*2.0060914855061203*cos(\t r)+1.0*2.0060914855061203*sin(\t r)});
\draw [shift={(15.999999999999993,1.000000000000007)},line width=2.0pt]  plot[domain=-2.3202107820269675:0.0,variable=\t]({1.0*2.000000000000007*cos(\t r)+-0.0*2.000000000000007*sin(\t r)},{0.0*2.000000000000007*cos(\t r)+1.0*2.000000000000007*sin(\t r)});
\draw [shift={(8.010764157855258,2.399146686545083)},line width=2.0pt]  plot[domain=4.618609473751706:5.13544622969049,variable=\t]({0.031228795445083578*11.60281936548135*cos(\t r)+-0.9995122622234555*10.072315903444755*sin(\t r)},{0.9995122622234555*11.60281936548135*cos(\t r)+0.031228795445083578*10.072315903444755*sin(\t r)});
\draw [shift={(15.08453928446209,6.849010726655737)},line width=2.0pt]  plot[domain=4.443267692073094:5.428252444126764,variable=\t]({-0.3913184311892814*3.233644093673443*cos(\t r)+-0.9202553370731189*2.1624545303792773*sin(\t r)},{0.9202553370731189*3.233644093673443*cos(\t r)+-0.3913184311892814*2.1624545303792773*sin(\t r)});
\draw [shift={(15.08453928446209,6.849010726655737)},line width=2.0pt]  plot[domain=-0.14246391152786675:1.747149063276433,variable=\t]({-0.3913184311892814*3.233644093673443*cos(\t r)+-0.9202553370731189*2.1624545303792773*sin(\t r)},{0.9202553370731189*3.233644093673443*cos(\t r)+-0.3913184311892814*2.1624545303792773*sin(\t r)});
\draw [shift={(8.010764157855258,2.399146686545083)},line width=2.0pt]  plot[domain=-0.5104239927679144:1.7188079367541982,variable=\t]({0.031228795445083578*11.60281936548135*cos(\t r)+-0.9995122622234555*10.072315903444755*sin(\t r)},{0.9995122622234555*11.60281936548135*cos(\t r)+0.031228795445083578*10.072315903444755*sin(\t r)});
\draw [shift={(12.037982850508092,10.697295259545422)},line width=2.0pt]  plot[domain=1.6147118343596758:1.9914795457959102,variable=\t]({-0.796755471056258*3.5461922194411764*cos(\t r)+-0.6043018445627322*2.058068182853418*sin(\t r)},{0.6043018445627322*3.5461922194411764*cos(\t r)+-0.796755471056258*2.058068182853418*sin(\t r)});
\draw [shift={(12.037982850508092,10.697295259545422)},line width=2.0pt]  plot[domain=2.4905157659517227:4.725219727865332,variable=\t]({-0.796755471056258*3.5461922194411764*cos(\t r)+-0.6043018445627322*2.058068182853418*sin(\t r)},{0.6043018445627322*3.5461922194411764*cos(\t r)+-0.796755471056258*2.058068182853418*sin(\t r)});
\draw [shift={(8.00184423223843,1.6155109705511068)},line width=2.0pt]  plot[domain=4.643344290921176:5.1976977495899535,variable=\t]({0.006175810708105875*8.384650832558677*cos(\t r)+-0.9999809294992068*6.0161770300302955*sin(\t r)},{0.9999809294992068*8.384650832558677*cos(\t r)+0.006175810708105875*6.0161770300302955*sin(\t r)});
\draw [shift={(8.00184423223843,1.6155109705511068)},line width=2.0pt]  plot[domain=-0.4977154342116883:1.4588217658016456,variable=\t]({0.006175810708105875*8.384650832558677*cos(\t r)+-0.9999809294992068*6.0161770300302955*sin(\t r)},{0.9999809294992068*8.384650832558677*cos(\t r)+0.006175810708105875*6.0161770300302955*sin(\t r)});
\draw (6.2,1.8) node[anchor=north west] {$\cdots$};
\draw (14.8,12) node[anchor=north west] {$1$};
\draw (19.5,7) node[anchor=north west] {$=$};
\draw (-0.9,-1.7) node[anchor=north west] {($n$ full left-hand twists)};
\end{tikzpicture}
\psset{xunit=0.53cm,yunit=0.53cm,algebraic=true,dimen=middle,dotstyle=o,dotsize=3pt 0,linewidth=0.8pt,arrowsize=3pt 2,arrowinset=0.25}
\begin{pspicture*}(-1.1,-3.2)(4.6,4.1)
\parametricplot[linewidth=2pt]{-4.946762204274805}{0.95392489822264}{-1*2*cos(t)+-0.01*1*sin(t)+1|0.01*2*cos(t)+-1*1*sin(t)+3}
\parametricplot[linewidth=2pt]{4.328585602502571}{5.149303185962962}{0.03*2*cos(t)+-1*1*sin(t)+1|1*2*cos(t)+0.03*1*sin(t)+1}
\parametricplot[linewidth=2pt]{-0.7625790658169525}{3.92520729608574}{0.03*2*cos(t)+-1*1*sin(t)+1|1*2*cos(t)+0.03*1*sin(t)+1}
\parametricplot[linewidth=2pt]{-0.8976578170467322}{4.986540318021191}{-1*2*cos(t)+-0.01*1*sin(t)+1|0.01*2*cos(t)+-1*1*sin(t)+-1}
\rput[tl](3.2,3.3){$-3$}
\rput[tl](2.2,1.5){$-\frac{1}{2}$}
\rput[tl](3.2,-0.5){$\frac{1-6n}{n}$}
\end{pspicture*}
\end{center}
\caption{Equivalent surgery presentations of the homology sphere}
\label{Figure:ComplicatedLinks}
\end{figure}

The following examples show how one can find an explicit lens space using the previous theorem.

\begin{example} {\rm Let $\frac{p}{q}=\frac{5}{2}$ and $\frac{r}{s}=\frac{109}{57}$. We have that $\frac{5}{2}=[2;2]$ and $\frac{109}{57}=[1;1,10,2,1+1]$, so Theorem \ref{Theorem:Basic} implies that
$$\frac{a}{b}=[-1;-2,-10,-1,-1,2,2]^{-1}-1=-\frac{431}{257}.$$
Since $L(p,q)\cong L(p,q')$ when $q\equiv q' \;(\mathrm{mod} \;p)$, we have that this manifold is also $L(431,174)$:}
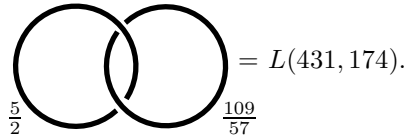
\psset{xunit=0.4cm,yunit=0.4cm,algebraic=true,dimen=middle,dotstyle=o,dotsize=3pt 0,linewidth=0.8pt,arrowsize=3pt 2,arrowinset=0.25}
\begin{center}
\begin{pspicture*}(-3.49,-1.5)(13,3.2)
\parametricplot[linewidth=2pt]{-3.758972595266308}{2.321589552819717}{1*1.98*cos(t)+0*1.98*sin(t)+4|0*1.98*cos(t)+1*1.98*sin(t)+1}
\parametricplot[linewidth=2pt]{-0.5735991128384956}{5.457219563177892}{1*2*cos(t)+0*2*sin(t)+1|0*2*cos(t)+1*2*sin(t)+1}
\rput[tl](-1.3,-0.2){$\frac{5}{2}$}
\rput[tl](5.8,-0.2){$\frac{109}{57}$}
\rput[tl](6.4,1.25){$=$}
\rput[tl](7.4,1.6){$L(431,174).$}
\end{pspicture*}
\end{center}
\end{example}

\begin{example} {\rm Let $\frac{p}{q}=\frac{3}{2}=[1;2]$ and $\frac{r}{s}=\frac{3}{4}=[0;1,3]$. Then, $\frac{a}{b}=[-2;-1,0,1,2]^{-1}-1=\frac{1}{0}$, hence
Theorem \ref{Theorem:Basic} implies that the resulting lens space is in fact $\mathbb{S}^3=L(1,0)$.}
\end{example}

This paper is organized as follows. In Section \ref{Section:Calculus} we describe rational (or Rolfsen) moves on a Hopf link in terms of continued fractions while in Section \ref{Section:Main} we prove the main theorem. We end with an application of the above theorem and prove a (well known) criterion for when a resulting lens space is $\mathbb{S}^3$ in terms of the framings $\frac{p}{q}$ and $\frac{r}{s}$ (Corollary \ref{cor: sphere}).

\section{Rational calculus} \label{Section:Calculus}

Kirby in \cite{Kir} showed that two integral framed links represent the same 3-manifold if and only if they are related by a finite sequence of moves of two kinds which are
called \textit{Kirby moves} (see \cite[VI.19]{PS96}). In the case of rational instead of integer framings, the analog of Kirby moves are \textit{Rolfsen moves} (see \cite{Rolf}), which we define below. Let $L=L_1\cup\hdots\cup L_n$ be an $n$-component link in $\mathbb{S}^3$, where $r_i\in\mathbb{Q}\cup \{\infty\}$ is a framing of $L_i$, for every $1\leq i\leq n$. A \textit{Rolfsen move of the first kind} states that we can add or delete component of the link with framing $\infty$. Let us suppose that the $i$-th component $L_i$ of $L$ is unknotted, and let $m$ be an integer. Then we may perform $m$ full twists along $L_i$, where by $m$ full twists we mean $m$ full right-hand twists if $m\geq 0$ and $|m|$ full left-hand twists if $m<0$. Then $L_i$, $r_i$, $L_j$ and $r_j$ $(j\neq i)$ change to $L_i'$, $r_i'$, $L_j'$ and $r_j'$, respectively, where
\begin{equation} \label{2Rolfsen}
r_i'=\frac{1}{m+1/r_i}, \qquad r_j'=r_j+m\mathrm{lk}(L_i,L_j)^2
\end{equation}
($1/0=\infty$ and $1/\infty=0$) and $\mathrm{lk}(L_i,L_j)$ is the linking number of $L_i$ and $L_j$. This is a \textit{Rolfsen move of the second kind}.

In this paper, we will perform Rolfsen moves only on Hopf links. Let us suppose that $L=L_1\cup L_2$ is a Hopf link, where $r_1=\frac{p}{q}$ is a framing of $L_1$, and $r_2=\frac{r}{s}$ is a framing of $L_2$ ($p,q,r,s\in\mathbb{Z}$). Since the absolute value of the linking number of the Hopf link is $1$, equations \eqref{2Rolfsen} give rise to 
$$r_1'=\frac{1}{m+q/p}=\frac{p}{q+mp}, \qquad r_2'=\frac{r}{s}+m=\frac{r+ms}{s},$$
where we have performed $m$ full twists along the component $L_1$ (see Figure \ref{Figure:2RolfsenHopf}).

\begin{figure}[!ht]
\psset{xunit=0.4cm,yunit=0.4cm,algebraic=true,dimen=middle,dotstyle=o,dotsize=3pt 0,linewidth=0.8pt,arrowsize=3pt 2,arrowinset=0.25}
\begin{center}
\begin{pspicture*}(-3.49,-2)(9,3.4)
\parametricplot[linewidth=2pt]{-3.758972595266308}{2.321589552819717}{1*1.98*cos(t)+0*1.98*sin(t)+4|0*1.98*cos(t)+1*1.98*sin(t)+1}
\parametricplot[linewidth=2pt]{-0.5735991128384956}{5.457219563177892}{1*2*cos(t)+0*2*sin(t)+1|0*2*cos(t)+1*2*sin(t)+1}
\rput[tl](-1.3,-0.7){$\frac{p}{q}$}
\rput[tl](5.6,-0.7){$\frac{r}{s}$}
\rput[tl](7.7,1.25){$=$}
\end{pspicture*}
\begin{pspicture*}(-2.2,-2)(7.5,3.4)
\parametricplot[linewidth=2pt]{-3.758972595266308}{2.321589552819717}{1*1.98*cos(t)+0*1.98*sin(t)+4|0*1.98*cos(t)+1*1.98*sin(t)+1}
\parametricplot[linewidth=2pt]{-0.5735991128384956}{5.457219563177892}{1*2*cos(t)+0*2*sin(t)+1|0*2*cos(t)+1*2*sin(t)+1}
\rput[tl](-2.2,-0.7){$\frac{p}{q+mp}$}
\rput[tl](5.4,-0.7){$\frac{r+ms}{s}$}
\end{pspicture*}
\end{center}
\caption{Rolfsen move of the second kind performed on the Hopf link}
\label{Figure:2RolfsenHopf}
\end{figure}
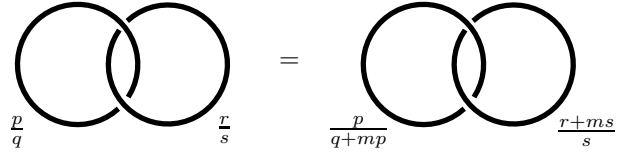

Note that using Rolfsen move of the second kind, the Hopf link with framing $(1,\frac{r}{s})$ is transformed to the Hopf link with framing $(\infty,\frac{r-s}{s})$, when we perform one full left-hand twist along the circle with framing $1$. Using the Rolfsen move of the first kind this is further transformed to the surgery along a circle with framing $\frac{r-s}{s}$, which is a surgery presentation of the lens space $L(r-s,s)$. Similarly, the Hopf link with framing $(\frac{p}{q},1)$ is transformed to a surgery presentation of the lens space $L(p-q,q)$.

Inspired by the Rolfsen move of the second kind, we define the following transformations of rational numbers. For a rational number $\frac{p}{q}\in \mathbb{Q}$ ($p$ and $q$ are coprime) and $m\in \mathbb{Z}$, we define operations
$$R_{1,m}\Big(\frac{p}{q}\Big)=\frac{p+mq}{q}, \qquad R_{-1,m}\Big(\frac{p}{q}\Big)=\frac{p}{q+mp}.$$

\begin{lem} \label{Lemma:EquivalentSurgery}
Let $m\in \mathbb{Z}$. Then,
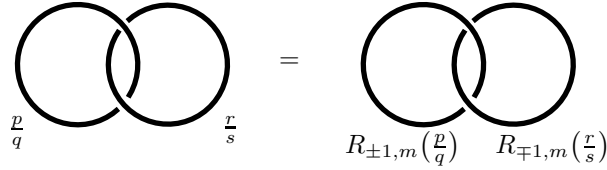
\psset{xunit=0.4cm,yunit=0.4cm,algebraic=true,dimen=middle,dotstyle=o,dotsize=3pt 0,linewidth=0.8pt,arrowsize=3pt 2,arrowinset=0.25}
\begin{center}
\begin{pspicture*}(-3.49,-2.5)(9,3.4)
\parametricplot[linewidth=2pt]{-3.758972595266308}{2.321589552819717}{1*1.98*cos(t)+0*1.98*sin(t)+4|0*1.98*cos(t)+1*1.98*sin(t)+1}
\parametricplot[linewidth=2pt]{-0.5735991128384956}{5.457219563177892}{1*2*cos(t)+0*2*sin(t)+1|0*2*cos(t)+1*2*sin(t)+1}
\rput[tl](-1.3,-0.7){$\frac{p}{q}$}
\rput[tl](5.8,-0.7){$\frac{r}{s}$}
\rput[tl](7.7,1.25){$=$}
\end{pspicture*}
\begin{pspicture*}(-2.2,-2.5)(7.5,3.4)
\parametricplot[linewidth=2pt]{-3.758972595266308}{2.321589552819717}{1*1.98*cos(t)+0*1.98*sin(t)+4|0*1.98*cos(t)+1*1.98*sin(t)+1}
\parametricplot[linewidth=2pt]{-0.5735991128384956}{5.457219563177892}{1*2*cos(t)+0*2*sin(t)+1|0*2*cos(t)+1*2*sin(t)+1}
\rput[tl](-1.6,-1.2){$R_{\pm 1,m}\big(\frac{p}{q}\big)$}
\rput[tl](3.4,-1.2){$R_{\mp 1,m}\big(\frac{r}{s}\big)$}
\end{pspicture*}
\end{center}
\end{lem}
\begin{proof}
This follows by using Rolfsen move of the second kind. 
\end{proof}

Note that we have the following straightforward property of continued fractions, which will be useful in the rest of the paper:
\begin{equation} \label{Equation:Zeros}
[\underbrace{0;0,\dots,0}_{i\text{ zeros}},a_i,\dots,a_l]=\begin{cases}[a_i;\dots, a_l],\quad&\text{ $i$ even},\\
[0;a_i,\dots,a_l], \quad &\text{ $i$ odd}.
\end{cases}
\end{equation}

\begin{lem} \label{Lemma:Kirby}
Let $[a_0;a_1,\hdots,a_n]$ be a continued fraction representation of $\frac{p}{q}$ and let $m$ be an integer. Then 
\begin{enumerate}
\item[(1)] $[a_0+m; a_1,\hdots,a_n]$ is a continued fraction representation of $R_{1,m}\big(\frac{p}{q}\big)$;
\item[(2)] $[0;m,a_0,\hdots,a_n]$ is a continued fraction representation of $R_{-1,m}\big(\frac{p}{q}\big)$.
\end{enumerate}
\end{lem}
\begin{proof}
Statements of the lemma follow directly from equations 
\begin{align*}
R_{1,m}\Big(\frac{p}{q}\Big)&=m+\frac{p}{q} \quad \text{and} \quad R_{-1,m}\Big(\frac{p}{q}\Big)=\frac{p}{q+mp}=\frac{1}{m+\frac{1}{\frac{p}{q}}}. \qedhere 
\end{align*}
\end{proof}

\begin{cor} \label{Corollary:KirbyComposition}
Let $[a_0;a_1,\hdots,a_l]$ be a continued fraction representation of $\frac{p}{q}$ and let $m_1,\dots,m_n$ be integers. Then the following equations hold
\begin{align*}
R_{(-1)^n,m_n}\circ\dots\circ R_{-1,m_1}
\Big(\frac{p}{q}\Big)
&=[m_n;\dots,m_1,a_0,\dots,a_l]^{(-1)^n}, \\
R_{(-1)^{n+1},m_n}\circ\dots\circ R_{1,m_1}
\Big(\frac{p}{q}\Big) &=[m_n;\dots,m_1+a_0,\dots,a_l]^{(-1)^{n+1}}.
\end{align*}
\end{cor}
\begin{proof}
Let us prove the first equality (the second being proved in a similar fashion). Using Lemma \ref{Lemma:Kirby}, we have
$$R_{-1,m_1}\Big(\frac{p}{q}\Big)=[0;m_1,a_0,\dots,a_l]
=[m_1,a_0,\dots,a_l]^{-1}.$$
For $n= 2$ we have
$$R_{1,m_2}\circ R_{-1,m_1}\Big(\frac{p}{q}\Big) =[m_2;m_1,a_0,\dots,a_l],$$
and the proof follows by induction on $n$.
\end{proof}

\begin{cor} \label{Corollary:Composition}
Let $[a_0;a_1,\dots,a_l]$ be a continued fraction representation of $\frac{p}{q}$ and $0\leq i \leq l-1$. Then 
$$R_{(-1)^{i},-a_i}\circ \dots \circ R_{1,-a_0} \Big( \frac{p}{q}\Big)=[\underbrace{0;0,\dots,0}_{i+1 \text{ zeros}},a_{i+1},\dots,a_l]=[a_{i+1};\dots,a_l]^{(-1)^{i+1}}.$$
\end{cor}
\begin{proof}
By Lemma \ref{Lemma:Kirby} we have $R_{1,-a_0}\big(\frac{p}{q}\big)=[0;a_1,\dots,a_l]$. For $i=1$, by using \eqref{Equation:Zeros} we obtain
\begin{align*}
R_{-1,-a_1}\circ R_{1,-a_0}\Big(\frac{p}{q}\Big)&=[0;-a_1,0,a_1,a_2,\dots,a_l]=[a_2;\dots,a_l]\\
&=[0;0,a_2,\dots,a_l]. 
\end{align*}
Now the claim of the corollary follows by induction on $i$.
\end{proof}

\begin{cor}
Let $[a_0;a_1,\dots,a_l+1]$ be a continued fraction representation of $\frac{p}{q}$ (note that the last term of this representation is written as $a_l+1$ and not as $a_l$ like in the previous statements). Then 
$$R_{(-1)^{l},-a_{l}}\circ \dots \circ R_{1,-a_0}\Big(\frac{p}{q}\Big)=1.$$
\end{cor}
\begin{proof}
Since for $l=0$, the statement of the corollary is trivial, let us suppose that $l>0$. By Corollary \ref{Corollary:Composition}, we have
$$R_{(-1)^{l-1},-a_{l-1}}\circ \dots \circ R_{1,-a_0}\Big(\frac{p}{q}\Big)=\begin{cases}
a_l+1,\quad &\text{ $l$ even}\\
\frac{1}{a_l+1},\quad & \text{ $l$ odd}.
\end{cases}$$
We finish the proof by noting that $R_{1,-a_l}(a_l+1)=1$ and $R_{-1,-a_l}\big(\frac{1}{a_l+1}\big)=1$. 
\end{proof}

\section{The main results}\label{Section:Main}

In this section we prove Theorem \ref{Theorem:Basic}.

\begin{prop} 
Let $\frac{p}{q}=[a_0;a_1,\hdots,a_n]$ and $\frac{r}{s}=[b_0;b_1,\hdots;b_m+1]$, and let $k_1,\dots, k_l\in \mathbb{Z}$. Then
\psset{xunit=0.4cm,yunit=0.4cm,algebraic=true,dimen=middle,dotstyle=o,dotsize=3pt 0,linewidth=0.8pt,arrowsize=3pt 2,arrowinset=0.25}
\begin{center}
\begin{pspicture*}(-3.49,-2.3)(9,3.4)
\parametricplot[linewidth=2pt]{-3.758972595266308}{2.321589552819717}{1*1.98*cos(t)+0*1.98*sin(t)+4|0*1.98*cos(t)+1*1.98*sin(t)+1}
\parametricplot[linewidth=2pt]{-0.5735991128384956}{5.457219563177892}{1*2*cos(t)+0*2*sin(t)+1|0*2*cos(t)+1*2*sin(t)+1}
\rput[tl](-1.3,-0.7){$\frac{p}{q}$}
\rput[tl](5.8,-0.7){$\frac{r}{s}$}
\rput[tl](7.7,1.25){$=$}
\end{pspicture*}
\begin{pspicture*}(-2.2,-2.3)(7.5,3.4)
\parametricplot[linewidth=2pt]{-3.758972595266308}{2.321589552819717}{1*1.98*cos(t)+0*1.98*sin(t)+4|0*1.98*cos(t)+1*1.98*sin(t)+1}
\parametricplot[linewidth=2pt]{-0.5735991128384956}{5.457219563177892}{1*2*cos(t)+0*2*sin(t)+1|0*2*cos(t)+1*2*sin(t)+1}
\rput[tl](-1.3,-0.7){$\frac{c}{d}$}
\rput[tl](5.8,-0.7){$\frac{e}{f}$}
\end{pspicture*}
\end{center}
where $\frac{c}{d}=[k_l;\dots,k_1,a_0,\dots,a_n]^{(-1)^l}$ and $\frac{e}{f}=[k_l;\dots,k_1+b_0,\dots,b_m]^{(-1)^{l+1}}$. In particular, 
\psset{xunit=0.4cm,yunit=0.4cm,algebraic=true,dimen=middle,dotstyle=o,dotsize=3pt 0,linewidth=0.8pt,arrowsize=3pt 2,arrowinset=0.25}
\begin{center}
\begin{pspicture*}(-3.49,-2.3)(9,3.4)
\parametricplot[linewidth=2pt]{-3.758972595266308}{2.321589552819717}{1*1.98*cos(t)+0*1.98*sin(t)+4|0*1.98*cos(t)+1*1.98*sin(t)+1}
\parametricplot[linewidth=2pt]{-0.5735991128384956}{5.457219563177892}{1*2*cos(t)+0*2*sin(t)+1|0*2*cos(t)+1*2*sin(t)+1}
\rput[tl](-1.3,-0.7){$\frac{p}{q}$}
\rput[tl](5.8,-0.7){$\frac{r}{s}$}
\rput[tl](7.7,1.25){$=$}
\end{pspicture*}
\begin{pspicture*}(-2.2,-2.3)(7.5,3.4)
\parametricplot[linewidth=2pt]{-3.758972595266308}{2.321589552819717}{1*1.98*cos(t)+0*1.98*sin(t)+4|0*1.98*cos(t)+1*1.98*sin(t)+1}
\parametricplot[linewidth=2pt]{-0.5735991128384956}{5.457219563177892}{1*2*cos(t)+0*2*sin(t)+1|0*2*cos(t)+1*2*sin(t)+1}
\rput[tl](-1.3,-0.7){$\frac{a}{b}$}
\rput[tl](5.8,-0.7){$1$}
\end{pspicture*}
\end{center}
where $\frac{a}{b}=[-b_m;\hdots,-b_0,a_0,
\hdots,a_n]^{(-1)^{m+1}}$.
\end{prop}

\begin{proof}
By Lemma \ref{Lemma:EquivalentSurgery}, the surgery $\big(\frac{p}{q},\frac{r}{s}\big)$ is equivalent to the surgery
$$\Big(R_{(-1)^l,k_l}\circ\dots\circ R_{-1,k_1}\Big(\frac{p}{q}\Big),R_{(-1)^{l+1},k_l}\circ\dots\circ R_{1,k_1}\Big(\frac{r}{s}\Big)\Big),$$
which after application of Corollary \ref{Corollary:KirbyComposition} yields the first part of the proposition. The rest of the statement follows when we take for $k_1,\hdots,k_l$ to be $-b_0,\hdots,-b_m$, respectively. 
\end{proof}

\begin{proof}[Proof of Theorem \ref{Theorem:Basic}] Proof of the main theorem now easily follows from the previous proposition. A surgery on a Hopf link with framing $\frac{p}{q}$ and $\frac{r}{s}$ is equivalent to a surgery on the same link with framing $\frac{a}{b}$ and $1$, where $\frac{a}{b}$ is given above. The last move we need to make is one twist in negative direction along the circle with framing $1$.
\end{proof}

\begin{cor}\label{cor: sphere}
Let $\frac{p}{q}, \frac{r}{s}\in \mathbb{Q}$ such that $pr-qs=\pm1$. Then, the $3$-manifold resulting from surgery on Hopf link with coefficients $\frac{p}{q}$ and $\frac{r}{s}$ is $\mathbb{S}^3$. 
\end{cor}
Before proving this result, we recall some of the well known facts about continued fractions. We say that a continued fraction $[a_0;a_1,\dots,a_n]$ is \textit{standard} if $a_1,\dots,a_n$ are positive integers and $a_n>1$. For such a standard continued fraction we put $\frac{p_i}{q_i}:=[a_0;a_1,\dots,a_i]$, $i=1,\dots,n$, assuming $p_i$ and $q_i$ to be coprime. Then, every rational number can be represented as a standard continued fraction in a unique way.

\begin{lem}\label{lem: rat func}
Let $[a_0;a_1,\dots,a_n]$ be a standard continued fraction and let $x$ be a variable. Then, we have an equality of rational functions:
\begin{equation}\label{eq: rat func}
[a_0;a_1,\dots,a_n,x]=\frac{p_nx+p_{n-1}}{q_nx+q_{n-1}},
\end{equation}
where the left-hand side of the previous equality is defined inductively as $a_0+\frac{1}{x}$ if $n=0$ and $[a_0;a_1,\dots,a_{n-1},\frac{a_n+1}{x}]$, otherwise.
\end{lem}
\begin{proof}
If we denote by $A(x)$ and $B(x)$ the left-hand and right-hand functions of \eqref{eq: rat func}, respectively, then the result can be proved in the same manner as Theorem 1 in \cite{Khin}.  
\end{proof}
\begin{lem}\label{lem: int solutions}
Let $\frac{p}{q}\in\mathbb{Q}$ ($p$ and $q$ coprime) and let $[a_0;a_1,\dots,a_n]$ be its standard continued fraction. If integers $x, y$ satisfy $px-qy=1$ (resp. $px-qy=-1$), then, $\frac{y}{x}=[a_0;a_1,\dots,a_n,l]$, for some integer $l$.
\end{lem}
\begin{proof}
One solution for $px-qy=1$ (resp. $px-qy=-1$) is given by $x_0=(-1)^{n+1}q_{n-1}$ and $y_0=(-1)^{n+1}p_{n-1}$ (resp. $x_0=(-1)^{n}q_{n-1}$ and $y_0=(-1)^{n}p_{n-1}$), by Theorem 2 in \cite{Khin}. Then, all the solutions are given by $x_0+lq_n$ and $y_0+lp_n$, for respective equations and respective $x_0$ and $y_0$. In particular, there exists an integer $l$, such that $\frac{y}{x}=\frac{p_nl+p_{n-1}}{q_nl+q_{n-1}}$. The result follows by Lemma \ref{lem: rat func}.
\end{proof}

\begin{proof}[Proof of Corollary \ref{cor: sphere}]
Let $[a_0;a_1,\dots,a_n]$ be standard continued fraction for $\frac{p}{q}$. By Lemma \ref{lem: int solutions} there exists an integer $l$ such that $\frac{r}{s}=[0;a_0,a_1,\dots,a_n,l]$. Hence, we have 
\begin{align*}
[-l+1;-a_n,\dots,-a_1,-a_0,0,&a_0,a_1,\dots,a_n]^{(-1)^{n}}-1\\
&=[-l+1;0]^{(-1)^n}-1 =\begin{cases}\frac{1}{0},\quad \text{\hspace{2mm}for $n$ even,}\\
\frac{1}{-1}, \quad \text{for $n$ odd.}
\end{cases}
\end{align*}
Now we can apply Theorem \ref{Theorem:Basic} and finish by noting that the resulting lens spaces $L(1,0)$ and $L(1,-1)$ are both homeomorphic to $\mathbb{S}^3$. 
\end{proof}

{\bf Acknowledgements.} The authors are grateful to Vladimir Gruji\' c and Zoran Petri\' c for suggesting the problem studied here, and for helpful discussions during the course of research and writing this paper. Also, we would like to thank the anonymous referee for useful comments which improved the presentation of the paper.

{\small
{\em Authors' addresses}:
{\em Velibor Bojkovi\' c}, Laboratoire de Math\' ematiques Nicolas Oresme, Caen, France, e-mail: \texttt{velibor.bojkovic@\allowbreak unicaen.fr};
{\em Jovana Nikoli\' c}, University of Belgrade, Faculty of Mathematics, Belgrade, Serbia, e-mail: \texttt{jovanadj@\allowbreak matf.bg.ac.rs};
{\em Mladen Zeki\' c} (corresponding author), Mathematical Institute of the Serbian Academy of Sciences and Arts, Belgrade, Serbia, e-mail: \texttt{mzekic@\allowbreak mi.sanu.ac.rs}.
}

\end{document}